\documentclass{amsart}%
\usepackage{amsfonts}
\usepackage{amsmath}
\usepackage{amssymb}
\usepackage{graphicx}%
\newtheorem{theorem}{Theorem}
\theoremstyle{plain}

\newtheorem{conjecture}{Conjecture}

\newtheorem{lemma}{Lemma}

\newtheorem{remark}{Remark}

\numberwithin{equation}{section}
\begin{document}
\title[Uniqueness Results Revisited]{Uniqueness Results on a geometric PDE in Riemannian and CR Geoemetry Revisited}
\author{Xiaodong Wang}
\address{Department of Mathematics, Michigan State University, East Lansing, MI 48824}
\email{xwang@math.msu.edu}

\begin{abstract}
We revisit some uniqueness results for a geometric nonlinear PDE related to
the scalar curvature in Riemannian geometry and CR geometry. In the Riemannian
case we give a new proof of the uniquness result assuming only a positive
lower bound for Ricci curvature. We apply the same principle in the CR case
and reconstruct the Jerison-Lee identity in a more general setting. As a
consequence we prove a stronger uniqueness result in the CR case. We also
discuss some open problems for further study.

\end{abstract}
\maketitle

\section{\bigskip Introduction}

Let $\left(  \Sigma^{n},g\right)  $ be a Riemannian manifold and
$\widetilde{g}=u^{4/\left(  n-2\right)  }g$ another metric conformal to $g$,
where $u$ is a positive smooth function on $\Sigma$. The scalar curvatures are
related by the following equation%
\[
-\frac{4\left(  n-1\right)  }{n-2}\Delta_{g}u+Ru=\widetilde{R}u^{\left(
n+2\right)  /\left(  n-2\right)  }.
\]
Let $\left(  \mathbb{S}^{n},g_{c}\right)  $ be the sphere with the standard
metric. A conformal metric $\widetilde{g}=u^{4/\left(  n-2\right)  }g_{c}$ has
constant scalar curvature $n\left(  n-1\right)  $ iff
\begin{equation}
-\frac{4}{n\left(  n-2\right)  }\Delta u+u=u^{\left(  n+2\right)  /\left(
n-2\right)  },\text{ on }\mathbb{S}^{n}. \label{sons}%
\end{equation}
Conformal diffeomorphisms of $\mathbb{S}^{n}$ give rise to a natural family of
solutions to the above equation%
\[
u_{t,\xi}\left(  x\right)  =\left(  \cosh t+\left(  \sinh t\right)  x\cdot
\xi\right)  ^{-\left(  n-2\right)  /2},
\]
where $t\geq0,\xi\in\mathbb{S}^{n}$. It is a remarkable theorem that these are
all the positive solutions to (\ref{sons}). There are now several proofs for
this theorem. Analytically, by the stereographic projection (\ref{sons}) is
equivalent to the following equation%
\[
-\Delta v=\frac{n\left(  n-2\right)  }{4}v^{\left(  n+2\right)  /\left(
n-2\right)  }\text{ on }\mathbb{R}^{n}%
\]
whose positive solutions were classified by Gidas-Ni-Nirenberg \cite{GNN}
using the moving plane method. Geometrically, it follows from the following
more general theorem of Obata.

\begin{theorem}
(\cite{O2}) Suppose $\left(  \Sigma^{n},\overline{g}\right)  $ is a closed
Einstein manifold and $g=\phi\overline{g}$ is a conformal metric with constant
scalar curvature, where $\phi$ is a positive smooth function. Then $\phi$ must
be constant unless $\left(  \Sigma^{n},\widetilde{g}\right)  $ is isometric to
the standard sphere $\left(  \mathbb{S}^{n},g_{c}\right)  $ up to a scaling
and $\phi$ corresponds to the following function on $\mathbb{S}^{n}$%
\[
\phi\left(  x\right)  =c\left(  \cosh t+\sinh tx\cdot a\right)  ^{-2}%
\]
for some $c>0,t\geq0$ and $a\in\mathbb{S}^{n}$.
\end{theorem}

Obata's proof is short and elegant and is based on the following formula
\[
\overline{T}=T+\left(  n-2\right)  \phi^{-1}\left(  D^{2}\phi-\frac{\Delta
\phi}{n}g\right)  ,
\]
where $T$ and $\overline{T}$ are the traceless Ricci tensor of $g$ and
$\overline{g}$, respectively. But this argument is quite subtle as it requires
using the unknown metric $g=\phi\overline{g}$ as the background metric instead
of the given Einstein metric $\overline{g}$.

There is parallel story in CR geometry. Let $M^{2m+1}$ be a CR manifold and
$\widetilde{\theta}=f^{2/m}\theta$ two pseudohermitian structures. The
pseudohermitian scalar curvatures of $\theta$ and $\widetilde{\theta}$ are
related by the following formula%
\[
-\frac{2\left(  m+1\right)  }{m}\Delta_{b}f+Rf=\widetilde{R}f^{\left(
m+2\right)  /m}.
\]
On $\mathbb{S}^{2m+1}=\left\{  z\in\mathbb{C}^{m+1}:\left\vert z\right\vert
=1\right\}  $ the canonical pseudohermitian structure $\theta_{c}=\left(
2\sqrt{-1}\overline{\partial}\left\vert z\right\vert ^{2}\right)
|_{\mathbb{S}^{2m+1}}$ satisfies $R_{\alpha\overline{\beta}}=\left(
m+1\right)  /2\delta_{\alpha\beta}$ and $R=m\left(  m+1\right)  /2$. Therefore
$\theta=f^{2/m}\theta_{c}$ has scalar curvature $m\left(  m+1\right)  /2$ iff
\begin{equation}
-\frac{4}{m^{2}}\Delta_{b}f+f=f^{\left(  m+2\right)  /m}\text{ on }%
\mathbb{S}^{2m+1}.\label{cres}%
\end{equation}
Pseudoconformal diffeomorphisms of $\mathbb{S}^{2m+1}$ yield a natural family
of solutions to the above equation%
\[
f_{t,\xi}\left(  z\right)  =\left\vert \cosh t+\left(  \sinh t\right)
z\cdot\overline{\xi}\right\vert ^{-1/m}.
\]
It is a remarkable result of Jerison and Lee \cite{JL} that these are all the
positive solutions of (\ref{cres}). The proof is based on a highly nontrivial
identity on $\left(  \mathbb{S}^{2m+1},\phi\theta_{c}\right)  $, $\phi
=f^{2/m}$
\begin{align}
&  \operatorname{Re}\left(  gD_{\alpha}+\overline{g}E_{\alpha}-3\phi_{0}%
\sqrt{-1}U_{\alpha}\right)  _{,\overline{\alpha}}\label{JL}\\
&  =\left(  \frac{1}{2}+\frac{1}{2}\phi\right)  \left(  \left\vert
D_{\alpha\beta}\right\vert ^{2}+\left\vert E_{\alpha\overline{\beta}%
}\right\vert ^{2}\right)  \nonumber\\
&  +\phi\left[  \left\vert D_{\alpha}-U_{\alpha}\right\vert ^{2}+\left\vert
U_{\alpha}+E_{\alpha}-D_{\alpha}\right\vert ^{2}+\left\vert U_{\alpha
}+E_{\alpha}\right\vert ^{2}+\left\vert \phi^{-1}\phi_{\overline{\gamma}%
}D_{\alpha\beta}+\phi^{-1}\phi_{\beta}E_{\alpha\overline{\gamma}}\right\vert
^{2}\right]  .\nonumber
\end{align}
where%
\begin{align*}
D_{\alpha\beta} &  =\phi^{-1}\phi_{\alpha,\beta},D_{\alpha}=\phi^{-1}%
\phi_{\overline{\beta}}D_{\alpha\beta},E_{\alpha}=\phi^{-1}\phi_{\gamma
}E_{\alpha\overline{\gamma}},\\
E_{\alpha\overline{\beta}} &  =\phi^{-1}\phi_{\alpha,\overline{\beta}}%
-\phi^{-2}\phi_{\alpha}\phi_{\overline{\beta}}-\frac{1}{2}\phi^{-1}\left(
g-\phi\right)  \delta_{\alpha\overline{\beta}},\\
U_{\alpha} &  =\frac{2}{m+2}D_{\alpha\beta,\overline{\beta}},g=\frac{1}%
{2}+\frac{1}{2}\phi+\phi^{-1}\left\vert \partial\phi\right\vert ^{2}+i\phi
_{0}.
\end{align*}
Here and throughout this paper we always work with a local unitary frame
$\{T_{\alpha}:\alpha=1,\cdots,m\}$ for $T^{1,0}M$ and $T_{0}=T$ is the Reeb
vector field. It should be emphasized that in all these formulas covariant
derivatives are calculated w.r.t. the unknown pseudoconformal structure
$\phi\theta_{c}$.

The Jerison-Lee identity is in fact valid on any closed Einstein
pseudohermitian manifold. Here by Einstein we mean $R_{\alpha\overline{\beta}%
}=\rho\delta_{\alpha\beta}$ and $A_{\alpha\beta}=0$ (torsion-free). The
following more general uniqueness result, which is the analogue of the Obata
theorem, was proved in \cite{W}.

\begin{theorem}
(\cite{W}) Let $\left(  M^{2m+1},\overline{\theta}\right)  $ be a closed
Einstein pseudohermitian manifold. Suppose $\theta=\phi\overline{\theta}$ is
another pseudohermitian structure with constant pseudohermitian scalar
curvature. Then $\phi$ must be constant unless $\left(  M^{2m+1}%
,\overline{\theta}\right)  $ is CR isometric to $\left(  \mathbb{S}%
^{2m+1},\theta_{c}\right)  $ up to a scaling and $\phi$ corresponds to the
following function on $\mathbb{S}^{2m+1}$%
\[
\phi\left(  z\right)  =c_{m}\left\vert \cosh t+\left(  \sinh t\right)
z\cdot\overline{\xi}\right\vert ^{-2}%
\]
for some $t\geq0$ and $\xi\in\mathbb{S}^{2m+1}$.
\end{theorem}

We note that like the Obata argument all calculations have to be carried out
with respect to the unknown pseudehermitian structure $\theta=\phi
\overline{\theta}$. Complicated formulas relating the curvature tensors of
$\theta$ and $\overline{\theta}$ as well as various Bianchi identities are
also heavily used in the proof.

The Jerison-Lee identity is truly remarkable and a better understanding is
highly desirable. In this paper, we propose a different approach to
reconstruct the formula. The basic idea is to study the model case carefully
and then come up with the right quantities to apply the maximum principle. We
first revisit the Riemannian case and give a new(?) proof of the uniqueness
results. In fact, this new proof does not require the Einstein condition. A
positive lower bound for Ricci curvature suffices. Suppose $\left(
M^{n},g\right)  $ is a compact Riemannian manifold with $Ric\geq n-1$ and
$u\in C^{\infty}\left(  M\right)  $ is \textbf{positive} and satisfies the
following equation%
\[
-\Delta u+\frac{n\left(  n-2\right)  }{4}u=\frac{n\left(  n-2\right)  }%
{4}u^{\left(  n+2\right)  /\left(  n-2\right)  }.
\]
If we write $u=v^{-\left(  n-2\right)  /2}$, then $v$ satisfies
\[
\Delta v=\frac{n}{2}v^{-1}\left(  \left\vert \nabla v\right\vert ^{2}%
+1-v^{2}\right)  .
\]
By the study of the model case, we consider \thinspace$\phi=v^{-1}\left(
\left\vert \nabla v\right\vert ^{2}+v^{2}+1\right)  $. A simple calculation
shows that
\[
\Delta\phi\geq\left(  n-2\right)  \left\langle \nabla\log v,\nabla
\phi\right\rangle
\]
and therefore the maximum principle comes into play. This simple argument
yields the following result which is more general than Obata's theorem.

\begin{theorem}
Let $\left(  M^{n},g\right)  $ be a smooth compact Riemannian manifold with a
(possibly empty) convex boundary. Suppose $u\in C^{\infty}\left(  M\right)  $
is a positive solution of the following equation%
\[%
\begin{array}
[c]{ccc}%
-\Delta u+\lambda u=u^{\left(  n+2\right)  /\left(  n-2\right)  } & \text{on}
& M,\\
\frac{\partial u}{\partial\nu}=0 & \text{on} & \partial M,
\end{array}
\]
where $\lambda>0$ is a constant. If $Ric\geq\left(  n-1\right)  g$ and
$\lambda\leq n\left(  n-2\right)  /4$, then $u$ must be constant unless
$\lambda=n\left(  n-2\right)  /4$ and $\left(  M,g\right)  $ is isometric to
$\left(  \mathbb{S}^{n},g_{c}\right)  $ or $\left(  \mathbb{S}_{+}^{n}%
,g_{c}\right)  $. In the latter case $u$ is given on $\mathbb{S}^{n}$ or
$\mathbb{S}_{+}^{n}$\ by the following formula%
\[
u(x)=c_{n}\left(  \cosh t+\left(  \sinh t\right)  x\cdot\xi\right)  ^{-\left(
n-2\right)  /2}.
\]
for some $t\geq0$ and $\xi\in\mathbb{S}^{n}$.
\end{theorem}

The above theorem is actually not new. It is a special case of a theorem by
Bidaut-V\'{e}ron and V\'{e}ron \cite{BVV} and Ilias \cite{I}. Their method is
based on a sophisticated integration by parts which can handle the subcritical
case as well. We will say more about their result in the last section.

We apply the same principle to the CR case. Here the first difficulty is that
there is no natural first order quantity and therefore we have to go to the
2nd order. There are three natural tensors of order 2 to consider and we must
take a suitable contraction and combination to apply the maximum principle. As
our argument in the Riemannian case, this approach has the advantage that the
calculations are done on a fixed pseudohermitian manifold $\left(
\Sigma^{2m+1},\theta\right)  $ which does not have to be Einstein. The unknown
pseudohermitian structure $\theta=\phi\widetilde{\theta}$ and its curvature
tensor do not enter the discussion at all. All it takes is to do covariant
derivatives. Of course we are using a lot of hindsight from Jerison-Lee.
Besides the identity (\ref{JL}) Jerison and Lee \cite{JL} gave three
additional divergence formulas on the Heisenberg group. The formula we obtain
can be viewed as the generalization of their first formula ((4.2) in
\cite{JL}) to any pseudohermitian manifold with torision zero. (One can even
drop this condition, but the additional terms involving the torsion
$A_{\alpha\beta}$ and its divergence seem too complicated). The calculations
are still formidable. But we hope that this approach sheds more light on the
Jerison-Lee work. We do get a more general identity, see Theorem \ref{JLn}. As
a result we prove a stronger uniqueness theorem.

\begin{theorem}
Let $\left(  M^{2m+1},\theta\right)  $ be a closed pseudohermitian manifold
with $A_{\alpha\beta}=0$ and $R_{\alpha\overline{\beta}}\geq\frac{m+1}{2}$.
Suppose $f>0$ satisfies the following equation on $M$%
\[
-\Delta_{b}f+\lambda f=f^{\left(  m+2\right)  /m},
\]
where $\lambda>0$ is a constant. If $\lambda\leq m^{2}/4$, then $f$ is
constant unless $\lambda=m^{2}/4$ and $\left(  M,\theta\right)  $ is isometric
to $\left(  \mathbb{S}^{2m+1},\theta_{c}\right)  $ and in this case
\[
f=c_{m}\left\vert \cosh t+\left(  \sinh t\right)  z\cdot\overline{\xi
}\right\vert ^{-1/m}%
\]
for some $t>0,\xi\in\mathbb{S}^{2m+1}$.
\end{theorem}

The paper is organized as follows. In the 2nd section we discuss the
Riemannian case. In Section 3 we study the model case in CR geometry as a
guide for finding the right quantities. In Section 4 we present our
reconstruction of the Jerison-Lee identity and prove the above uniqueness
result. We discuss some open problems in the last section.

\section{The Riemannian case}

On $\left(  \mathbb{S}^{n},g_{c}\right)  $ we consider the equation
\begin{equation}
-\frac{4}{n\left(  n-2\right)  }\Delta u+u=u^{\left(  n+2\right)  /\left(
n-2\right)  }. \label{ss2}%
\end{equation}
If $u$ is positive, the equation simply means that $u^{4/\left(  n-2\right)
}g_{c}$ has the same scalar curvature $n\left(  n-1\right)  $. For $t\geq
0,\xi\in\mathbb{S}^{n}$ the map $\Phi_{t,\xi}:\mathbb{S}^{n}\rightarrow
\mathbb{S}^{n}$ defined by%
\[
\Phi_{t,\xi}\left(  x\right)  =\frac{1}{\cosh t+\sinh t\left(  x\cdot
\xi\right)  }\left(  x-\left(  x\cdot\xi\right)  \xi\right)  +\frac{\sinh
t+\cosh t\left(  x\cdot\xi\right)  }{\cosh t+\sinh t\left(  x\cdot\xi\right)
}\xi
\]
is a conformal diffeomorphism with $\Phi_{t,\xi}^{\ast}g_{c}=u_{t,\xi
}^{4/\left(  n-2\right)  }g_{c}$ with%

\[
u_{t,\xi}\left(  x\right)  =\left(  \cosh t+\left(  \sinh t\right)  x\cdot
\xi\right)  ^{-\left(  n-2\right)  /2}.
\]
Therefore these are solutions of the equation (\ref{ss2}).

If we write $u=v^{-\left(  n-2\right)  /2}$, then $v=\cosh t+\left(  \sinh
t\right)  x\cdot\xi$. We compute
\begin{align*}
\left\vert \nabla v\right\vert ^{2}  &  =\sinh^{2}t\left\vert \nabla\left(
x\cdot\xi\right)  \right\vert ^{2}\\
&  =\sinh^{2}t\left(  1-\left(  x\cdot\xi\right)  ^{2}\right) \\
&  =\sinh^{2}t-\left(  v-\cosh t\right)  ^{2}\\
&  =-1-v^{2}+2v\cosh t.
\end{align*}
It follows that $v^{1}\left(  \left\vert \nabla v\right\vert ^{2}%
+v^{2}+1\right)  =2\cosh t$ is a constant.

Suppose now $\left(  M^{n},g\right)  $ is a compact Riemannian manifold with
$Ric\geq n-1$ and $u\in C^{\infty}\left(  M\right)  $ is \textbf{positive} and
satisfies the following equation%
\begin{equation}
-\Delta u+\frac{n\left(  n-2\right)  }{4}u=\frac{n\left(  n-2\right)  }%
{4}u^{\left(  n+2\right)  /\left(  n-2\right)  }. \label{eom}%
\end{equation}
If $Ric=n-1$ as in Obata's theorem, the above equation simply means that the
scalar curvature of $\widetilde{g}:=u^{4/\left(  n-2\right)  }$ equals
$n\left(  n-1\right)  $. In the following discussion, this geometric
interpretation plays no role. We write $u=v^{-\left(  n-2\right)  /2}$. By
direct calculation, $v>0$ satisfies the following equation%
\[
\Delta v=\frac{n}{2}v^{-1}\left(  \left\vert \nabla v\right\vert ^{2}%
+1-v^{2}\right)  .
\]
In view of the model case, we set $\phi=v^{-1}\left(  \left\vert \nabla
v\right\vert ^{2}+v^{2}+1\right)  $. The above equation becomes $\Delta
v+nv=\frac{n}{2}\phi$. As $v\phi=\left\vert \nabla v\right\vert ^{2}+v^{2}+1$,
we compute using the Bochner formula%
\begin{align*}
&  \frac{1}{2}v\Delta\phi+\left\langle \nabla v,\nabla\phi\right\rangle
+\frac{1}{2}\phi\Delta v\\
&  =\left\vert D^{2}v\right\vert ^{2}+\left\langle \nabla v,\nabla\Delta
v\right\rangle +Ric\left(  \nabla v,\nabla v\right)  +v\Delta v+\left\vert
\nabla v\right\vert ^{2}\\
&  \geq\frac{\left(  \Delta v\right)  ^{2}}{n}+\left\langle \nabla
v,\nabla\Delta v\right\rangle +n\left\vert \nabla v\right\vert ^{2}+v\Delta
v\\
&  =\frac{\Delta v}{n}\left(  \Delta v+nv\right)  +\left\langle \nabla
v,\nabla\left(  \Delta v+nv\right)  \right\rangle \\
&  =\frac{1}{2}\phi\Delta v+\frac{n}{2}\left\langle \nabla v,\nabla
\phi\right\rangle .
\end{align*}
Therefore we obtain
\[
\Delta\phi\geq\left(  n-2\right)  \left\langle \nabla\log v,\nabla
\phi\right\rangle .
\]
If $\partial M\neq\emptyset$, direct calculation yields the following formula
for the normal derivative along $\partial M$
\[
\frac{1}{2}\frac{\partial\phi}{\partial\nu}=f^{-1}\left[  \chi\left(
D^{2}v\left(  \nu,\nu\right)  +f-\frac{1}{2}\phi\right)  +\left\langle \nabla
f,\nabla\chi\right\rangle -\Pi\left(  \nabla f,\nabla f\right)  \right]  ,
\]
where $f=v|_{\partial M},\chi=\frac{\partial v}{\partial\nu}$ and $\Pi$ is the
2nd fundamental form. By these calculations, we can now deduce the following
uniqueness result.

\begin{theorem}
\label{otr}Let $\left(  M^{n},g\right)  $ be a smooth compact Riemannian
manifold with a (possibly empty) convex boundary. Suppose $u\in C^{\infty
}\left(  M\right)  $ is a positive solution of the following equation%
\[%
\begin{array}
[c]{ccc}%
-\Delta u+\lambda u=u^{\left(  n+2\right)  /\left(  n-2\right)  } & \text{on}
& M,\\
\frac{\partial u}{\partial\nu}=0 & \text{on} & \partial M,
\end{array}
\]
where $\lambda>0$ is a constant. If $Ric\geq\left(  n-1\right)  g$ and
$\lambda\leq n\left(  n-2\right)  /4$, then $u$ must be constant unless
$\lambda=n\left(  n-2\right)  /4$ and $\left(  M,g\right)  $ is isometric to
$\left(  \mathbb{S}^{n},g_{c}\right)  $ or $\left(  \mathbb{S}_{+}^{n}%
,g_{c}\right)  $. In the latter case $u$ is given on $\mathbb{S}^{n}$ or
$\mathbb{S}_{+}^{n}$\ by the following formula%
\[
u(x)=c_{n}\left(  \cosh t+\left(  \sinh t\right)  x\cdot\xi\right)  ^{-\left(
n-2\right)  /2}.
\]
for some $t\geq0$ and $\xi\in\mathbb{S}^{n}$.
\end{theorem}

\begin{proof}
We first take $\lambda=n\left(  n-2\right)  /4$. By scaling $u$ we can
consider the equivalent equation (\ref{eom}). Then the above calculations for
the assoicated $v$ and $\phi$ yield%
\[
\Delta\phi\geq\left(  n-2\right)  \left\langle \nabla\log v,\nabla
\phi\right\rangle \text{ on }M;\frac{\partial\phi}{\partial\nu}\leq0\text{ on
}\partial M
\]
under our assumptions. By the maximum principle and Hopf lemma, $\phi$ must be
a constant. Inspecting the proof shows that we must have $D^{2}v=\frac{1}%
{n}\Delta vg$. If $v$ is not constant, it is easy to deduce from this
over-determined system that $\left(  M,g\right)  $ is isometric to $\left(
\mathbb{S}^{n},g_{c}\right)  $ or $\left(  \mathbb{S}_{+}^{n},g_{c}\right)  $
and $v$ is up to a constant a linear function. This finishes the proof when
$\lambda=n\left(  n-2\right)  /4$.

When $\lambda<n\left(  n-2\right)  /4$, we consider the scaled metric
$\widetilde{g}=cg$. Then $u$ satisfies
\[
-\widetilde{\Delta}u+c^{-1}\lambda u=c^{-1}u^{\left(  n+2\right)  /\left(
n-2\right)  }.
\]
We choose $c=\frac{4\lambda}{n\left(  n-2\right)  }<1$. As $Ric\left(
\widetilde{g}\right)  \geq\frac{n-1}{c}\widetilde{g}>\left(  n-1\right)
\widetilde{g}$, we can apply the result for $\lambda=n\left(  n-2\right)  /4$
on $\left(  M^{n},\widetilde{g}\right)  $.
\end{proof}

\section{The CR sphere}

Consider the unit sphere $\mathbb{S}^{2m+1}=\left\{  z\in\mathbb{C}%
^{m+1}:\left\vert z\right\vert =1\right\}  $ with the canonical
pseudohermitian structure
\[
\theta_{c}=\left(  2\sqrt{-1}\overline{\partial}\left\vert z\right\vert
^{2}\right)  |_{\mathbb{S}^{2m+1}}=2\sum_{i=1}^{m+1}\left(  x_{i}dy_{i}%
-y_{i}dx_{i}\right)  ,
\]
i.e. $\theta_{c}\left(  X\right)  =2\left\langle J\xi,X\right\rangle $ at
$\xi\in\mathbb{S}^{2m+1}$. Then%
\[
d\theta_{c}\left(  X,Y\right)  =4\left\langle JX,Y\right\rangle .
\]
The Reeb vector field is
\begin{align*}
T  &  =\frac{\sqrt{-1}}{2}\sum_{i=1}^{m+1}\left(  z_{i}\frac{\partial
}{\partial z_{i}}-\overline{z}_{i}\frac{\partial}{\partial\overline{z}_{i}%
}\right) \\
&  =\frac{1}{2}\sum_{i=1}^{m+1}\left(  x_{i}\frac{\partial}{\partial y_{i}%
}-y_{i}\frac{\partial}{\partial x_{i}}\right) \\
&  =\frac{1}{2}J\xi.
\end{align*}
Therefore the adapted metric $g_{c}=4g_{0}$, where $g_{0}$ is the standard
metric on $\mathbb{S}^{2m+1}$. We have $R_{\alpha\overline{\beta}}=\left(
m+1\right)  /2\delta_{\alpha\beta},R=m\left(  m+1\right)  /2$. We now consider
the following equation%
\begin{equation}
-\frac{4}{m^{2}}\Delta_{b}f+f=f^{\left(  m+2\right)  /m}\text{ on }%
\mathbb{S}^{2m+1}. \label{sscr2}%
\end{equation}
If $f>0$, it simply means the pseudohermitian structure $f^{2/m}\theta_{c}$
has the same constant scalar curvature $m\left(  m+1\right)  /2$.

\bigskip For $t\geq0,\xi\in\mathbb{S}^{n}$ the map $\Phi_{t,\xi}%
:\mathbb{S}^{2m+1}\rightarrow\mathbb{S}^{2m+1}$ defined by%
\[
\Phi_{t,\xi}\left(  z\right)  =\frac{1}{\cosh t+\sinh t\left(  z\cdot
\overline{\xi}\right)  }\left(  z-\left(  z\cdot\overline{\xi}\right)
\xi\right)  +\frac{\sinh t+\cosh t\left(  z\cdot\overline{\xi}\right)  }{\cosh
t+\sinh t\left(  z\cdot\overline{\xi}\right)  }\xi
\]
is a pseudoconformal diffeomorphism with $\Phi_{t,\xi}^{\ast}\theta
_{c}=f_{t,\xi}^{2/m}\theta_{c}$, where%
\[
f_{t,\xi}\left(  z\right)  =\left\vert \cosh t+\left(  \sinh t\right)
z\cdot\overline{\xi}\right\vert ^{-1/m}.
\]
Therefore these are solutions to the equation (\ref{sscr2}). We write such a
solution as $f=\phi^{-m/2}$. Then%
\[
\phi\left(  z\right)  =\left\vert \cosh t+\left(  \sinh t\right)
z\cdot\overline{\xi}\right\vert ^{2}.
\]
We want to see what identities $\phi$ satisfy. In the following, we always
take $\theta_{c}$ and its adapted metric $g_{c}=4g_{0}$ as a background
metric. With $f\left(  z\right)  =z\cdot\overline{\xi}$, we have
\[
\phi=\cosh^{2}t+\sinh^{2}t\left\vert f\right\vert ^{2}+\sinh t\cosh t\left(
f+\overline{f}\right)  .
\]
As $f_{\overline{\alpha}}=0$,
\[
\phi_{\alpha}=\sinh t\left(  \cosh t+\sinh t\overline{f}\right)  f_{\alpha}.
\]
We also observe, as $T\left(  z\right)  =\frac{1}{2}Jz$,
\[
f_{0}=\frac{1}{2}\frac{d}{dt}|_{t=0}f\left(  e^{ir}z\right)  =\frac{\sqrt{-1}%
}{2}f.
\]
If we write $f=u+\sqrt{-1}v$, then $u_{0}=-\frac{1}{2}v$. Moreover,
$4\left\vert \nabla u\right\vert ^{2}=1-u^{2}$ (1st eigenfunction). As a first
eigenfunction, we have $D^{2}f=-\frac{1}{4}fg_{\theta}$. It follows
\begin{align*}
f_{\alpha,\beta}  &  =0,\\
f_{\alpha,\overline{\beta}}  &  =D^{2}f\left(  T_{\alpha},T_{\overline{\beta}%
}\right)  +\frac{\sqrt{-1}}{2}f_{0}\delta_{\alpha\beta}\\
&  =\left(  -\frac{1}{4}f+\frac{\sqrt{-1}}{2}f_{0}\right)  \delta_{\alpha
\beta}\\
&  =-\frac{1}{2}f\delta_{\alpha\beta}.
\end{align*}

\begin{lemma}
We have
\[
\phi^{-1}\left\vert \partial\phi\right\vert ^{2}=\frac{1}{2}\sinh t\left(
1-\left\vert f\right\vert ^{2}\right)  .
\]

\end{lemma}

\begin{proof}
As $f_{\overline{\alpha}}=0$, we have $f_{\alpha}=2u_{\alpha}$. We compute%
\begin{align*}
\left\vert \partial\phi\right\vert ^{2}  &  =\phi\sinh^{2}t\left\vert \partial
f\right\vert ^{2}\\
&  =4\phi\sinh^{2}t\left\vert \partial u\right\vert ^{2}\\
&  =2\phi\sinh^{2}t\left(  \left\vert \nabla u\right\vert ^{2}-u_{0}%
^{2}\right) \\
&  =\frac{1}{2}\phi\sinh^{2}t\left(  1-u^{2}-v^{2}\right) \\
&  =\frac{1}{2}\phi\sinh^{2}t\left(  1-\left\vert f\right\vert ^{2}\right)  .
\end{align*}

\end{proof}

Let $g=\frac{1}{2}+\frac{1}{2}\phi+\phi^{-1}\left\vert \partial\phi\right\vert
^{2}+\sqrt{-1}\phi_{0}$.

\begin{lemma}
The function $\phi$ satisfies the following three tensor equations
\begin{align*}
\phi_{\alpha,\beta}  &  =0,\\
\phi_{\alpha,\overline{\beta}}-\phi^{-1}\phi_{\alpha}\phi_{\overline{\beta}}
&  =\frac{1}{2}\left(  g-\phi\right)  \delta_{\alpha\beta}.\\
\phi_{0,\alpha}  &  =\frac{\sqrt{-1}}{2}\phi^{-1}\overline{g}\phi_{\alpha}%
\end{align*}

\end{lemma}

\begin{proof}
The 1st identity is obvious. We have $\phi_{0}=\frac{i}{2}\sinh t\cosh
t\left(  f-\overline{f}\right)  $. Thus%
\begin{align*}
g  &  =\frac{1}{2}+\frac{1}{2}\phi+\frac{1}{2}\sinh^{2}t\left(  1-\left\vert
f\right\vert ^{2}\right)  -\frac{1}{2}\sinh t\cosh t\left(  f-\overline
{f}\right) \\
&  =\cosh^{2}t+\sinh t\cosh t\overline{f}\\
&  =\cosh t\left(  \cosh t+\sinh t\overline{f}\right)  .
\end{align*}
We compute%
\begin{align*}
\phi_{\alpha,\overline{\beta}}  &  =\sinh^{2}tf_{\alpha}\overline{f_{\beta}%
}-\frac{1}{2}f\sinh t\left(  \cosh t+\sinh t\overline{f}\right)
\delta_{\alpha\beta}\\
&  =\phi^{-1}\phi_{\alpha}\phi_{\overline{\beta}}-\frac{1}{2}\sinh t\left(
\cosh tf+\sinh t\left\vert f\right\vert ^{2}\right)  \delta_{\alpha\beta}%
\end{align*}
To finish the proof of the 2nd identity, we check
\begin{align*}
g-\phi &  =-\sinh^{2}t\left\vert f\right\vert ^{2}-\sinh t\cosh tf\\
&  =-\sinh t\left(  \cosh tf+\sinh t\left\vert f\right\vert ^{2}\right)  .
\end{align*}
The last identity follows from
\[
\phi_{0,\alpha}=\frac{i}{2}\sinh t\cosh tf_{\alpha}.
\]

\end{proof}

\begin{remark}
There is an additional identity
\[
\phi_{0,0}=\frac{1}{2}\phi^{-1}\left(  \frac{1-\phi^{2}}{4}+\phi
^{-1}\left\vert \partial\phi\right\vert ^{2}+\phi^{-2}\left\vert \partial
\phi\right\vert ^{4}+\phi_{0}^{2}\right)
\]
which is of higher order as it involves the 2nd order derivative in the
direction of the Reeb vector field.
\end{remark}

\section{Reconstructing the Jerison-Lee identity}

We now consider the general case. Let $\left(  M^{2m+1},\theta\right)  $ be a
closed pseudohermitian manifold with torsion $A_{\alpha\beta}=0$. Suppose
$f\in C^{\infty}\left(  M\right)  $ is a positive solution of the following
equation%
\[
-\frac{4}{m^{2}}\Delta_{b}f+f=f^{\left(  m+2\right)  /m}\text{ }.
\]
If $\theta$ has scalar curvature $m\left(  m+1\right)  /2\,$, the equation
simply means that $\widetilde{\theta}=f^{2/m}\theta$ has also scalar curvature
$m\left(  m+1\right)  /2$. But this interpretation does not play any role in
our discussion. Let $\phi=f^{-2/m}$. The above equation becomes
\begin{equation}
\phi_{\alpha,\overline{\alpha}}=\frac{m}{2}i\phi_{0}+\frac{m}{4}\left(
1-\phi\right)  +\frac{m+2}{2}\phi^{-1}\left\vert \partial\phi\right\vert ^{2}.
\label{GE}%
\end{equation}
Motivated by the study in the model case, we introduce%
\begin{align*}
B_{\alpha\overline{\beta}}  &  =\phi_{\alpha,\overline{\beta}}-\phi^{-1}%
\phi_{\alpha}\phi_{\overline{\beta}}-\frac{1}{2}\left(  \frac{1}{2}-\frac
{1}{2}\phi+\phi^{-1}\left\vert \partial\phi\right\vert ^{2}+i\phi_{0}\right)
\delta_{\alpha\overline{\beta}},\\
C_{\alpha}  &  =i\phi_{0,\alpha}+\frac{1}{2}\phi^{-1}\left(  \frac{1}{2}%
+\frac{1}{2}\phi+\phi^{-1}\left\vert \partial\phi\right\vert ^{2}-i\phi
_{0}\right)  \phi_{\alpha}%
\end{align*}
Note that the equation (\ref{GE}) simply means $B_{\alpha\overline{\alpha}}%
=0$. As the study of the model case suggests, to prove the uniqueness result
we must prove $\phi_{\alpha\beta}=0,B_{\alpha\overline{\beta}}$ and
$C_{\alpha}=0$. We set
\[
g=\frac{1}{2}+\frac{1}{2}\phi+\phi^{-1}\left\vert \partial\phi\right\vert
^{2}+i\phi_{0}.
\]
Then we can rewrite these equations as
\begin{align*}
B_{\alpha\overline{\beta}}  &  =\phi_{\alpha,\overline{\beta}}-\phi^{-1}%
\phi_{\alpha}\phi_{\overline{\beta}}-\frac{1}{2}\left(  g-\phi\right)
\delta_{\alpha\beta},\\
C_{\alpha}  &  =i\phi_{0,\alpha}+\frac{1}{2}\phi^{-1}\overline{g}\phi_{\alpha
},\\
\phi_{\alpha\overline{\alpha}}  &  =\frac{m}{2}\left(  g-\phi\right)
+\phi^{-1}\left\vert \partial\phi\right\vert ^{2}%
\end{align*}
We further introduce the contractions%
\[
A_{\alpha}=\phi_{\alpha\beta}\phi_{\overline{\beta}},B_{\alpha}=B_{\alpha
\overline{\beta}}\phi_{\beta}.
\]
Therefore we have three complex $\left(  1,0\right)  $ vector fields
$A_{\alpha},B_{\alpha}$ and $C_{\alpha}$. Their conjugates will be denoted by
$A_{\overline{\alpha}},B_{\overline{\alpha}}$ and $C_{\overline{\alpha}}$. Our
first goal is to calculate the divergence for these vector fields. We need
some preliminary formulas.

\begin{lemma}
\label{basicid}We have\ \ \
\begin{align*}
\left(  \phi^{-1}\left\vert \partial\phi\right\vert ^{2}\right)
_{\overline{\alpha}}  &  =\phi^{-1}\left(  A_{\overline{\alpha}}%
+B_{\overline{\alpha}}\right)  +\frac{1}{2}\phi^{-1}\left(  g-\phi\right)
\phi_{\overline{\alpha}},\\
g_{\overline{\alpha}}  &  =\phi^{-1}\left(  A_{\overline{\alpha}}%
+B_{\overline{\alpha}}\right)  -C_{\overline{\alpha}}+\phi^{-1}g\phi
_{\overline{\alpha}},\\
\overline{g}_{\overline{\alpha}}  &  =\phi^{-1}\left(  A_{\overline{\alpha}%
}+B_{\overline{\alpha}}\right)  +C_{\overline{\alpha}}%
\end{align*}

\end{lemma}

\begin{proof}
We compute%
\begin{align*}
\left(  \phi^{-1}\left\vert \partial\phi\right\vert ^{2}\right)
_{\overline{\alpha}}  &  =\phi^{-1}\left(  \phi_{\beta}\phi_{\overline{\beta
},\overline{\alpha}}+\phi_{\beta,\overline{\alpha}}\phi_{\overline{\beta}%
}\right)  -\phi^{-2}\left\vert \partial\phi\right\vert ^{2}\phi_{\overline
{\alpha}}\\
&  =\phi^{-1}\left(  A_{\overline{\alpha}}+\overline{\phi_{\alpha
,\overline{\beta}}\phi_{\beta}}+i\phi_{0}\phi_{\overline{\alpha}}\right)
-\phi^{-2}\left\vert \partial\phi\right\vert ^{2}\phi_{\overline{\alpha}}.
\end{align*}
Eliminating $\phi_{\alpha,\overline{\beta}}$ using the formula for
$B_{\alpha\overline{\beta}}$ yields%
\begin{align*}
\left(  \phi^{-1}\left\vert \partial\phi\right\vert ^{2}\right)
_{\overline{\alpha}}=  &  \phi^{-1}\left(  A_{\overline{\alpha}}%
+\overline{B_{\alpha\overline{\beta}}\phi_{\beta}+\phi^{-1}\left\vert
\partial\phi\right\vert ^{2}\phi_{\alpha}+\frac{1}{2}\left(  g-\phi\right)
\phi_{\alpha}}+i\phi_{0}\phi_{\overline{\alpha}}\right) \\
&  -\phi^{-2}\left\vert \partial\phi\right\vert ^{2}\phi_{\overline{\alpha}%
}.\\
=  &  \phi^{-1}\left(  A_{\overline{\alpha}}+B_{\overline{\alpha}}+\frac{1}%
{2}\left(  \overline{g}+2i\phi_{0}-\phi\right)  \phi_{\overline{\alpha}%
}\right) \\
=  &  \phi^{-1}\left(  A_{\overline{\alpha}}+B_{\overline{\alpha}}+\frac{1}%
{2}\left(  g-\phi\right)  \phi_{\overline{\alpha}}\right)  .
\end{align*}
This proves the 1st identity. Differentiating $g$ yields%
\begin{align*}
g_{\overline{\alpha}}  &  =\frac{1}{2}\phi_{\overline{\alpha}}+\left(
\phi^{-1}\left\vert \partial\phi\right\vert ^{2}\right)  _{\overline{\alpha}%
}+i\phi_{0,\overline{\alpha}}\\
&  =\phi^{-1}\left(  A_{\overline{\alpha}}+B_{\overline{\alpha}}\right)
+\frac{1}{2}\phi^{-1}g\phi_{\overline{\alpha}}+i\phi_{0,\overline{\alpha}}\\
&  =\phi^{-1}\left(  A_{\overline{\alpha}}+B_{\overline{\alpha}}\right)
-C_{\overline{\alpha}}+\phi^{-1}g\phi_{\overline{\alpha}}%
\end{align*}
Differentiating $\overline{g}$ yields%
\begin{align*}
\overline{g}_{\overline{\alpha}}  &  =\phi^{-1}\left(  A_{\overline{\alpha}%
}+B_{\overline{\alpha}}\right)  +\frac{1}{2}\phi^{-1}g\phi_{\overline{\alpha}%
}-i\phi_{0,\overline{\alpha}}\\
&  =\phi^{-1}\left(  A_{\overline{\alpha}}+B_{\overline{\alpha}}\right)
+C_{\overline{\alpha}}.
\end{align*}

\end{proof}

\begin{lemma}
\bigskip\label{div2}We have%
\begin{align*}
\phi_{\alpha,\beta\overline{\beta}}  &  =\frac{m+2}{2}\left[  \phi^{-1}\left(
A_{\alpha}+B_{\alpha}\right)  +C_{\alpha}\right]  +R_{\alpha\overline{\beta}%
}\phi_{\beta}-\frac{m+1}{2}\phi_{\alpha},\\
B_{\alpha\overline{\beta},\overline{\alpha}}  &  =\frac{\left(  m-1\right)
}{2}\phi^{-1}A_{\overline{\beta}}+\frac{m+1}{2}\phi^{-1}B_{\overline{\beta}%
}-\frac{\left(  m-1\right)  }{2}C_{\overline{\beta}}.
\end{align*}

\end{lemma}

\begin{proof}
We compute, using Lemma \ref{basicid}%
\begin{align*}
\phi_{\alpha,\beta\overline{\beta}}=  &  \phi_{\beta,\overline{\beta}\alpha
}+i\phi_{\alpha,0}+R_{\alpha\overline{\beta}}\phi_{\beta}\\
=  &  \frac{m}{2}\left(  g_{\alpha}-\phi_{\alpha}\right)  +\left(  \phi
^{-1}\left\vert \partial\phi\right\vert ^{2}\right)  _{\alpha}+i\phi
_{\alpha,0}+R_{\alpha\overline{\beta}}\phi_{\beta}\\
=  &  \frac{m}{2}\left[  \phi^{-1}\left(  A_{\alpha}+B_{\alpha}\right)
+C_{\alpha}-\phi_{\alpha}\right]  +\phi^{-1}\left(  A_{\alpha}+B_{\alpha
}\right)  +\frac{1}{2}\phi^{-1}\left(  \overline{g}-\phi\right)  \phi_{\alpha
}\\
&  +i\phi_{\alpha,0}+R_{\alpha\overline{\beta}}\phi_{\beta}\\
=  &  \frac{m}{2}\left[  \phi^{-1}\left(  A_{\alpha}+B_{\alpha}\right)
+C_{\alpha}\right]  +\phi^{-1}\left(  A_{\alpha}+B_{\alpha}\right)  +\frac
{1}{2}\phi^{-1}\overline{g}\phi_{\alpha}+i\phi_{\alpha,0}\\
&  +R_{\alpha\overline{\beta}}\phi_{\beta}-\frac{m+1}{2}\phi_{\alpha}\\
=  &  \frac{m+2}{2}\left[  \phi^{-1}\left(  A_{\alpha}+B_{\alpha}\right)
+C_{\alpha}\right]  +R_{\alpha\overline{\beta}}\phi_{\beta}-\frac{m+1}{2}%
\phi_{\alpha}.
\end{align*}
Similarly, using the equation of $\phi$%
\begin{align*}
B_{\alpha\overline{\beta},\overline{\alpha}}=  &  \phi_{\alpha,\overline
{\beta}\overline{\alpha}}-\phi^{-1}\left(  \phi_{\alpha,\overline{\alpha}}%
\phi_{\overline{\beta}}+\phi_{\alpha}\phi_{\overline{\beta},\overline{\alpha}%
}\right)  +\phi^{-2}\left\vert \partial\phi\right\vert ^{2}\phi_{\overline
{\beta}}-\frac{1}{2}\left(  g_{\overline{\beta}}-\phi_{\overline{\beta}%
}\right) \\
=  &  \phi_{\alpha,\overline{\alpha}\overline{\beta}}-\phi^{-1}\left(
\phi_{\alpha,\overline{\alpha}}\phi_{\overline{\beta}}+A_{\overline{\beta}%
}\right)  +\phi^{-2}\left\vert \partial\phi\right\vert ^{2}\phi_{\overline
{\beta}}-\frac{1}{2}\left(  g_{\overline{\beta}}-\phi_{\overline{\beta}%
}\right) \\
=  &  \frac{m-1}{2}\left(  g_{\overline{\beta}}-\phi_{\overline{\beta}%
}\right)  +\left(  \phi^{-1}\left\vert \partial\phi\right\vert ^{2}\right)
_{\overline{\beta}}-\frac{m}{2}\phi^{-1}\left(  g-\phi\right)  \phi
_{\overline{\beta}}-\phi^{-1}A_{\overline{\beta}}.
\end{align*}
Using Lemma \ref{basicid}, we obtain%
\begin{align*}
B_{\alpha\overline{\beta},\overline{\alpha}}=  &  \frac{m-1}{2}\left[
\phi^{-1}\left(  A_{\overline{\beta}}+B_{\overline{\beta}}\right)
-C_{\overline{\beta}}+\phi^{-1}\left(  g-\phi\right)  \phi_{\overline{\beta}%
}\right] \\
&  +\left[  \phi^{-1}\left(  A_{\overline{\beta}}+B_{\overline{\beta}}\right)
+\frac{1}{2}\phi^{-1}\left(  g-\phi\right)  \phi_{\overline{\beta}}\right]
-\frac{m}{2}\phi^{-1}\left(  g-\phi\right)  \phi_{\overline{\beta}}-\phi
^{-1}A_{\overline{\beta}}\\
=  &  \frac{m-1}{2}\left[  \phi^{-1}\left(  A_{\overline{\beta}}%
+B_{\overline{\beta}}\right)  -C_{\overline{\beta}}\right]  +\phi
^{-1}B_{\overline{\beta}}\\
=  &  \frac{m-1}{2}\left(  \phi^{-1}A_{\overline{\beta}}-C_{\overline{\beta}%
}\right)  +\frac{m+1}{2}\phi^{-1}B_{\overline{\beta}}.
\end{align*}

\end{proof}

We are ready to calculated the divergence for the three vector fields
$A_{\alpha},B_{\alpha}$ and $C_{\alpha}$.

\begin{lemma}
\bigskip\label{divM}We have%
\begin{align*}
A_{\alpha,\overline{\alpha}}=  &  \frac{m+2}{2}\left[  C_{\alpha}+\phi
^{-1}\left(  A_{\alpha}+B_{\alpha}\right)  \right]  \phi_{\overline{\alpha}%
}+\left\vert \phi_{\alpha\beta}\right\vert ^{2}+Q,\\
B_{\alpha,\overline{\alpha}}=  &  \left[  \frac{\left(  m-1\right)  }{2}%
\phi^{-1}A_{\overline{\alpha}}+\frac{m+1}{2}\phi^{-1}B_{\overline{\alpha}%
}-\frac{\left(  m-1\right)  }{2}C_{\overline{\alpha}}\right]  \phi_{\alpha
}+\phi^{-1}B_{\alpha}\phi_{\overline{\alpha}}+\left\vert B_{\alpha
\overline{\beta}}\right\vert ^{2}\\
C_{\alpha,\overline{\alpha}}=  &  \frac{m+2}{2}\phi^{-1}C_{\alpha}%
\phi_{\overline{\alpha}}-\frac{m+1}{2}\phi^{-1}C_{\overline{\alpha}}%
\phi_{\alpha}+\frac{1}{2}\phi^{-2}\left(  A_{\overline{\alpha}}+B_{\overline
{\alpha}}\right)  \phi_{\alpha}+S,
\end{align*}
where
\begin{align*}
Q  &  =R_{\alpha\overline{\beta}}\phi_{\beta}\phi_{\overline{\alpha}}%
-\frac{m+1}{2}\left\vert \partial\phi\right\vert ^{2}\\
S  &  =-\frac{m}{2}\left[  \phi_{0,0}-\frac{1}{2}\phi^{-1}\left(  \frac
{1-\phi^{2}}{4}+\phi^{-1}\left\vert \partial\phi\right\vert ^{2}+\phi
^{-2}\left\vert \partial\phi\right\vert ^{4}+\phi_{0}^{2}\right)  \right]
\end{align*}

\end{lemma}

\begin{proof}
The 1st two identities follow directly from Lemma \ref{basicid} and Lemma
\ref{div2}. To prove the 3rd identity, we first note $\phi_{0,\alpha}%
=\phi_{\alpha,0}$ and $\phi_{\alpha,0\overline{\beta}}=\phi_{\alpha
,\overline{\beta}0}$ as we assume that $\left(  M,\theta\right)  $ is
torsion-free. We compute%
\begin{align*}
\left(  \phi^{-1}\left\vert \partial\phi\right\vert ^{2}\right)  _{0}=  &
\phi^{-1}\left(  \phi_{\beta}\phi_{\overline{\beta},0}+\phi_{\beta,0}%
\phi_{\overline{\beta}}\right)  -\phi^{-2}\left\vert \partial\phi\right\vert
^{2}\phi_{0}\\
=  &  \phi^{-1}\phi_{\beta}\left(  iC_{\overline{\beta}}-\frac{i}{2}\phi
^{-1}g\phi_{\overline{\beta}}\right)  +\phi^{-1}\phi_{\overline{\beta}}\left(
-iC_{\beta}+\frac{i}{2}\phi^{-1}\overline{g}\phi_{\beta}\right) \\
&  -\phi^{-2}\left\vert \partial\phi\right\vert ^{2}\phi_{0}\\
=  &  i\phi^{-1}\left(  \phi_{\beta}C_{\overline{\beta}}-\phi_{\overline
{\beta}}C_{\beta}\right)  +\frac{i}{2}\phi^{-2}\left\vert \partial
\phi\right\vert ^{2}\left(  \overline{g}-g\right)  -\phi^{-2}\left\vert
\partial\phi\right\vert ^{2}\phi_{0}\\
=  &  i\phi^{-1}\left(  \phi_{\beta}C_{\overline{\beta}}-\phi_{\overline
{\beta}}C_{\beta}\right)  .
\end{align*}
Using these identities as well as the previous lemmas, we compute%
\begin{align*}
C_{\alpha,\overline{\alpha}}=  &  i\phi_{\alpha,\overline{\alpha}0}+\frac
{1}{2}\phi^{-1}\overline{g}_{\overline{\alpha}}\phi_{\alpha}+\frac{1}{2}%
\phi^{-1}\overline{g}\phi_{\alpha\overline{\alpha}}-\frac{1}{2}\phi
^{-2}\overline{g}\left\vert \partial\phi\right\vert ^{2}\\
=  &  i\left[  \frac{m}{2}\left(  i\phi_{00}-\frac{1}{2}\phi_{0}\right)
+\frac{m+2}{2}\left(  \phi^{-1}\left\vert \partial\phi\right\vert ^{2}\right)
_{0}\right]  +\frac{1}{2}\phi^{-1}\phi_{\alpha}\left[  \phi^{-1}\left(
A_{\overline{\alpha}}+B_{\overline{\alpha}}\right)  +C_{\overline{\alpha}%
}\right] \\
&  +\frac{m}{4}\phi^{-1}\overline{g}\left(  g-\phi\right) \\
=  &  -\frac{m}{4}i\phi_{0}-\frac{m}{2}\phi_{00}-\frac{m+2}{2}\phi^{-1}\left(
\phi_{\alpha}C_{\overline{\alpha}}-\phi_{\overline{\alpha}}C_{\alpha}\right)
\\
&  +\frac{1}{2}\phi^{-1}\phi_{\alpha}\left[  \phi^{-1}\left(  A_{\overline
{\alpha}}+B_{\overline{\alpha}}\right)  +C_{\overline{\alpha}}\right]
+\frac{m}{4}\phi^{-1}\overline{g}\left(  g-\phi\right) \\
=  &  \frac{m+2}{2}\phi^{-1}\phi_{\overline{\alpha}}C_{\alpha}-\frac{m+1}%
{2}\phi^{-1}\phi_{\alpha}C_{\overline{\alpha}}+\frac{1}{2}\phi^{-2}%
\phi_{\alpha}\left(  A_{\overline{\alpha}}+B_{\overline{\alpha}}\right) \\
&  -\frac{m}{2}\phi_{00}+\frac{m}{4}\phi^{-1}\left\vert g\right\vert
^{2}-\frac{m}{4}\left(  \overline{g}+i\phi_{0}\right) \\
=  &  \frac{m+2}{2}\phi^{-1}\phi_{\overline{\alpha}}C_{\alpha}-\frac{m+1}%
{2}\phi^{-1}\phi_{\alpha}C_{\overline{\alpha}}+\frac{1}{2}\phi^{-2}%
\phi_{\alpha}\left(  A_{\overline{\alpha}}+B_{\overline{\alpha}}\right) \\
&  -\frac{m}{2}\left[  \phi_{00}-\frac{1}{2}\phi^{-1}\left\vert g\right\vert
^{2}+\frac{1}{2}\left(  \frac{1}{2}+\frac{1}{2}\phi+\phi^{-1}\left\vert
\partial\phi\right\vert ^{2}\right)  \right]  .
\end{align*}
Using the formula for $g$ in the last term yields the 3rd identity.
\end{proof}

These three formulas demonstrate that the three vector fields $A_{\alpha
},B_{\alpha}$ and $C_{\alpha}$ intertwine with one another. The basic
strategy, following Jerison-Lee, is to come up with a linear combination of
the three vectors fields whose divergence is a sum of squares. The divergence
of $C_{\alpha}$ is much more complicated as the last term $S$ involves 2nd
order derivatives in the $T$ direction which we don't know how to control. But
fortunately this extra term is purely real. Another simple but crucial fact
for the final formula is the following.

\begin{lemma}
\label{real}$B_{\alpha}\phi_{\overline{\alpha}}$ is real, i.e. $B_{\alpha}%
\phi_{\overline{\alpha}}=B_{\overline{\alpha}}\phi_{\alpha}$.
\end{lemma}

\begin{proof}
We have%
\begin{align*}
B_{\alpha\overline{\beta}}  &  =\phi_{\alpha,\overline{\beta}}-\frac{i}{2}%
\phi_{0}\delta_{\alpha\beta}-\phi^{-1}\phi_{\alpha}\phi_{\overline{\beta}%
}-\frac{1}{2}\left(  \operatorname{Re}g-\phi\right)  \delta_{\alpha\beta}\\
&  =\frac{1}{2}\left(  \phi_{\alpha,\overline{\beta}}+\phi_{\overline{\alpha
},\beta}\right)  -\phi^{-1}\phi_{\alpha}\phi_{\overline{\beta}}-\frac{1}%
{2}\left(  \operatorname{Re}g-\phi\right)  \delta_{\alpha\beta}.
\end{align*}
It is clearly hermitian. Therefore $B_{\alpha}\phi_{\overline{\alpha}%
}=B_{\alpha\overline{\beta}}\phi_{\beta}\phi_{\overline{\alpha}}$ is real.
\end{proof}

We can now state the final formula, which can be viewed as a generalization of
the formula (4.2) on the Heisenberg group in \cite{JL}.

\begin{theorem}
\label{JLn}Let $\left(  M^{2m+1},\theta\right)  $ is a closed pseudohermitian
manifold with torsion $A_{\alpha\beta}=0$. Suppose $\phi\in C^{\infty}\left(
M\right)  $ is positive and satisfies the equation (\ref{GE}). Then
\begin{align*}
&  \operatorname{Re}\left[  \left(  \phi^{-\left(  m+1\right)  }\left(
\left(  \overline{g}+3i\phi_{0}\right)  \phi^{-1}A_{\alpha}+\left(
\overline{g}-i\phi_{0}\right)  \phi^{-1}B_{\alpha}-i3\phi_{0}C_{\alpha
}\right)  \right)  _{\overline{\alpha}}\right] \\
=  &  \phi^{-\left(  m+1\right)  }\left[  \left(  \frac{1}{2}+\frac{1}{2}%
\phi\right)  \left(  \left\vert \phi_{\alpha\beta}\right\vert ^{2}+\left\vert
B_{\alpha\overline{\beta}}\right\vert ^{2}\right)  +\left(  \frac{1}{2}%
+\frac{1}{2}\phi+\phi^{-1}\left\vert \partial\phi\right\vert ^{2}\right)
Q\right. \\
&  \left.  +\left\vert \phi^{-1}A_{\alpha}-C_{\alpha}\right\vert
^{2}+\left\vert \phi^{-1}B_{\alpha}+C_{\alpha}\right\vert ^{2}+\left\vert
C_{\alpha}\right\vert ^{2}+\phi^{-1}\left\vert \phi_{\alpha\beta}%
\phi_{\overline{\gamma}}+B_{\alpha\overline{\gamma}}\phi_{\beta}\right\vert
^{2}\right]  ,
\end{align*}
where
\[
Q=R_{\alpha\overline{\beta}}\phi_{\beta}\phi_{\overline{\alpha}}-\frac{m+1}%
{2}\left\vert \partial\phi\right\vert ^{2}.
\]

\end{theorem}

\begin{proof}
By direct calculation, the divergence on the LHS is given by%
\begin{align*}
&  \phi^{-\left(  m+1\right)  }\left[  \left(  \overline{g}+3i\phi_{0}\right)
\left(  \phi^{-1}A_{\alpha,\overline{\alpha}}-\left(  m+2\right)  \phi
^{-2}A_{\alpha}\phi_{\overline{\alpha}}\right)  +\phi^{-1}\left(  \overline
{g}_{\overline{\alpha}}+3i\phi_{0,\overline{\alpha}}\right)  A_{\alpha}\right.
\\
&  +\left(  \overline{g}-i\phi_{0}\right)  \left(  \phi^{-1}B_{\alpha
,\overline{\alpha}}-\left(  m+2\right)  \phi^{-2}B_{\alpha}\phi_{\overline
{\alpha}}\right)  +\phi^{-1}\left(  \overline{g}_{\overline{\alpha}}%
-i\phi_{0,\overline{\alpha}}\right)  B_{\alpha}\\
&  \left.  -i3\left(  \phi_{0}C_{\alpha,\overline{\alpha}}+\phi_{0,\overline
{\alpha}}C_{\alpha}-\left(  m+1\right)  \phi^{-1}\phi_{0}C_{\alpha}%
\phi_{\overline{\alpha}}\right)  \right]  .
\end{align*}
Using Lemma as well as the formula $-i\phi_{0,\overline{\alpha}}%
=C_{\overline{\alpha}}-\frac{1}{2}\phi^{-1}g\phi_{\overline{\alpha}}$ the
above expression is (ignoring the factor $\phi^{-\left(  m+1\right)  }$)%
\begin{align*}
&  \left(  \overline{g}+3i\phi_{0}\right)  \left(  \frac{m+2}{2}\phi
^{-1}C_{\alpha}\phi_{\overline{\alpha}}-\frac{m+2}{2}\phi^{-2}A_{\alpha}%
\phi_{\overline{\alpha}}+\frac{m+2}{2}\phi^{-2}B_{\alpha}\phi_{\overline
{\alpha}}+\left\vert \phi_{\alpha\beta}\right\vert ^{2}+Q\right) \\
&  +\phi^{-1}\left(  \phi^{-1}\left(  A_{\overline{\alpha}}+B_{\overline
{\alpha}}\right)  -2C_{\overline{\alpha}}+\frac{3}{2}\phi^{-1}g\phi
_{\overline{\alpha}}\right)  A_{\alpha}\\
&  +\left(  \overline{g}-i\phi_{0}\right)  \left(  -\frac{m-1}{2}\phi
^{-1}C_{\overline{\alpha}}\phi_{\alpha}+\frac{m-1}{2}\phi^{-2}A_{\overline
{\alpha}}\phi_{\alpha}-\frac{m+1}{2}\phi^{-2}B_{\overline{\alpha}}\phi
_{\alpha}+\left\vert B_{\alpha\overline{\beta}}\right\vert ^{2}\right) \\
&  +\phi^{-1}\left(  \phi^{-1}\left(  A_{\overline{\alpha}}+B_{\overline
{\alpha}}\right)  +2C_{\overline{\alpha}}-\frac{1}{2}\phi^{-1}g\phi
_{\overline{\alpha}}\right)  B_{\alpha}\\
&  -i3\left(  \phi_{0}\left(  -\frac{m}{2}\phi^{-1}C_{\alpha}\phi
_{\overline{\alpha}}-\frac{m+1}{2}\phi^{-1}C_{\overline{\alpha}}\phi_{\alpha
}+\frac{1}{2}\phi^{-2}\left(  A_{\overline{\alpha}}+B_{\overline{\alpha}%
}\right)  \phi_{\alpha}\right)  \right) \\
&  +3\left(  C_{\overline{\alpha}}-\frac{1}{2}\phi^{-1}g\phi_{\overline
{\alpha}}\right)  C_{\alpha}-3i\phi_{0}S.
\end{align*}
After expansions and cancelations, we arrive at%
\begin{align*}
&  \left(  \overline{g}+3i\phi_{0}\right)  \left(  \left\vert \phi
_{\alpha\beta}\right\vert ^{2}+Q\right)  +\left(  \overline{g}-i\phi
_{0}\right)  \left\vert B_{\alpha\overline{\beta}}\right\vert ^{2}+\phi
^{-2}\left\vert A_{\alpha}+B_{\alpha}\right\vert ^{2}+2\phi^{-1}\left(
B_{\alpha}-A_{\alpha}\right)  C_{\overline{\alpha}}\\
&  +3\left\vert C_{\alpha}\right\vert ^{2}+\phi^{-2}E_{1}+\phi^{-2}E_{2}%
+\phi^{-1}E_{3}-3i\phi_{0}S,
\end{align*}
where
\begin{align*}
E_{1}  &  =\frac{\left(  m-1\right)  }{2}\left(  A_{\overline{\alpha}}%
\phi_{\alpha}-A_{\alpha}\phi_{\overline{\alpha}}\right)  \operatorname{Re}%
g-\left(  m+\frac{1}{2}\right)  i\phi_{0}\left(  A_{\alpha}\phi_{\overline
{\alpha}}+A_{\overline{\alpha}}\phi_{\alpha}\right)  ,\\
E_{2}  &  =\left(  2m+1\right)  i\phi_{0}B_{\alpha}\phi_{\overline{\alpha}}\\
E_{3}  &  =\frac{\left(  m-1\right)  }{2}\left(  C_{\alpha}\phi_{\overline
{\alpha}}-C_{\overline{\alpha}}\phi_{\alpha}\right)  \operatorname{Re}%
g+\frac{5m+1}{2}i\phi_{0}\left(  C_{\alpha}\phi_{\overline{\alpha}%
}+C_{\overline{\alpha}}\phi_{\alpha}\right)  .
\end{align*}
It is clear that all these three terms as well as the last term (recall $S$ is
real) are purely imaginary. Therefore
\begin{align*}
&  \operatorname{Re}\left[  \left(  \phi^{-\left(  m+1\right)  }\left(
\left(  \overline{g}+3i\phi_{0}\right)  \phi^{-1}A_{\alpha}+\left(
\overline{g}-i\phi_{0}\right)  \phi^{-1}B_{\alpha}-i3\phi_{0}C_{\alpha
}\right)  \right)  _{\overline{\alpha}}\right] \\
=  &  \phi^{-\left(  m+1\right)  }\operatorname{Re}\left[  \left(
\overline{g}+3i\phi_{0}\right)  \left(  \left\vert \phi_{\alpha\beta
}\right\vert ^{2}+Q\right)  +\left(  \overline{g}-i\phi_{0}\right)  \left\vert
B_{\alpha\overline{\beta}}\right\vert ^{2}+\phi^{-2}\left\vert A_{\alpha
}+B_{\alpha}\right\vert ^{2}\right. \\
&  \left.  -2\phi^{-1}\left(  A_{\alpha}+B_{\alpha}\right)  C_{\overline
{\alpha}}+3\left\vert C_{\alpha}\right\vert ^{2}\right] \\
=  &  \phi^{-\left(  m+1\right)  }\left[  \left(  \frac{1}{2}+\frac{1}{2}%
\phi+\phi^{-1}\left\vert \partial\phi\right\vert ^{2}\right)  \left(
\left\vert \phi_{\alpha\beta}\right\vert ^{2}+\left\vert B_{\alpha
\overline{\beta}}\right\vert ^{2}+Q\right)  +\phi^{-2}\left\vert A_{\alpha
}+B_{\alpha}\right\vert ^{2}\right. \\
&  \left.  +3\left\vert C_{\alpha}\right\vert ^{2}+2\phi^{-1}\operatorname{Re}%
\left(  B_{\alpha}-A_{\alpha}\right)  C_{\overline{\alpha}}\right]  .
\end{align*}
It is then elementary to show that this equals the RHS.
\end{proof}

\begin{remark}
The divergence formula (4.2) in \bigskip\cite{JL} has been used recently by Ma
and Ou \cite{MO} to show that the following equation on the Heisenberg group
$\mathbb{H}^{m}$%
\[
-\Delta_{b}u=u^{q},
\]
where $q<\left(  m+2\right)  /m$, has no positive solution.
\end{remark}

With this identity, we can now prove the following.

\begin{theorem}
\label{otcr}Let $\left(  M^{2m+1},\theta\right)  $ be a closed pseudohermitian
manifold with $A_{\alpha\beta}=0$ and $R_{\alpha\overline{\beta}}\geq
\frac{m+1}{2}$. Suppose $f>0$ satisfies the following equation on $M$%
\[
-\Delta_{b}f+\lambda f=f^{\left(  m+2\right)  /m},
\]
where $\lambda>0$ is a constant. If $\lambda\leq m^{2}/4$, then $f$ is
constant unless $\lambda=m^{2}/4$ and $\left(  M,\theta\right)  $ is isometric
to $\left(  \mathbb{S}^{2m+1},\theta_{c}\right)  $ and in this case
\[
f=c_{m}\left\vert \cosh t+\left(  \sinh t\right)  z\cdot\overline{\xi
}\right\vert ^{-1/m}%
\]
for some $t>0,\xi\in\mathbb{S}^{2m+1}$.
\end{theorem}

It suffices to prove it for $\lambda=m^{2}/4$ as the case when $\lambda
<m^{2}/4$ then follows by scaling $\theta$ as in the proof of Theorem
\ref{otr}. Therefore in the following we assume $\lambda=m^{2}/4$. By scaling
$f$ we consider the equivalent equation $\frac{4}{m^{2}}\Delta_{b}%
f+f=f^{\left(  m+2\right)  /m}$. By our previous discussion, $\phi:=f^{-m/2}$
satisfies the identity in Theorem \ref{JLn}. Under our assumption, $Q\geq0$.
The RHS of the identity is nonnegative while the LHS is a divergence.
Therefore, all the terms on the RHS must vanish. In particular,
\[
\phi_{\alpha\beta}=0,B_{\alpha\overline{\beta}}=0,R_{\alpha\overline{\beta}%
}\phi_{\beta}=\frac{m+1}{2}\phi_{\alpha}.
\]
The arguments in \cite{W} can then be applied to finish the proof.

\section{Some open problems}

In the introduction, we have alluded to a more general uniqueness result in
the Riemannian case, which can be stated as follows.

\begin{theorem}
\label{vi}(\cite{BVV, I} ) Let $\left(  M^{n},g\right)  $ be a smooth compact
Riemannian manifold with a (possibly empty) convex boundary. Suppose $u\in
C^{\infty}\left(  M\right)  $ is a positive solution of the following equation%
\[%
\begin{array}
[c]{ccc}%
-\Delta u+\lambda u=u^{q} & \text{on} & M,\\
\frac{\partial u}{\partial\nu}=0 & \text{on} & \partial M,
\end{array}
\]
where $\lambda>0$ is a constant and $1<q\leq\left(  n+2\right)  /\left(
n-2\right)  $. If $Ric\geq\left(  n-1\right)  g$ and $\lambda\left(
q-1\right)  \leq n$, then $u$ must be constant unless $q=\left(  n+2\right)
/\left(  n-2\right)  ,\lambda=n\left(  n-2\right)  /4$ and $\left(
M,g\right)  $ is isometric to $\left(  \mathbb{S}^{n},\frac{4\lambda}{n\left(
n-2\right)  }g_{\mathbb{S}^{n}}\right)  $ or $\left(  \mathbb{S}_{+}^{n}%
,\frac{4\lambda}{n\left(  n-2\right)  }g_{\mathbb{S}^{n}}\right)  $. In the
latter case $u$ is given on $\mathbb{S}^{n}$ or $\mathbb{S}_{+}^{n}$\ by the
following formula%
\[
u(x)=c_{n}\left(  \cosh t+\left(  \sinh t\right)  x\cdot\xi\right)  ^{-\left(
n-2\right)  /2}.
\]
for some $t\geq0$ and $\xi\in\mathbb{S}^{n}$.
\end{theorem}

Our Theorem \ref{otr} corresponds to the special case $q=\left(  n+2\right)
/\left(  n-2\right)  $. It would be interesting if the method presented in
Section 2 can be sharpened to give a new proof of the above theorem in it full generality.

It is natural to wonder if there is a CR analogue of Theorem \ref{vi}. We are
attempted to make the following conjecture.

\begin{conjecture}
Let $\left(  M^{2m+1},\theta\right)  $ be a closed pseudohermitian manifold
with $A_{\alpha\beta}=0$ and $R_{\alpha\overline{\beta}}\geq\frac{m+1}{2}$.
Suppose $f>0$ satisfies the following equation on $M$%
\[
-\Delta_{b}f+\lambda f=f^{q},
\]
where $\lambda>0$ is a constant and $1<q<\left(  m+2\right)  /m$. If
$\lambda\left(  q-1\right)  \leq m/2$, then $f$ is constant unless $q=\left(
m+2\right)  /m,\lambda=m^{2}/4$ and $\left(  M,\theta\right)  $ is isometric
to $\left(  \mathbb{S}^{2m+1},\theta_{c}\right)  $ and in this case
\[
f=c_{m}\left\vert \cosh t+\left(  \sinh t\right)  z\cdot\overline{\xi
}\right\vert ^{-1/m}%
\]
for some $t>0,\xi\in\mathbb{S}^{2m+1}$.
\end{conjecture}

The case $q=\left(  m+2\right)  /m$ is exactly our Theorem \ref{otcr}.

The Conjecture, if true, would imply the following sharp Sobolev inequality on
any closed pseudohermitian manifold $\left(  M^{2m+1},\theta\right)  $ with
$A_{\alpha\beta}=0$ and $R_{\alpha\overline{\beta}}\geq\frac{m+1}{2}$: for
\thinspace$1\leq q\leq\left(  m+2\right)  /m$%
\[
\left(  \frac{1}{V}\int_{M}\left\vert F\right\vert ^{q+1}dv\right)
^{2/\left(  q+1\right)  }\leq\frac{2\left(  q-1\right)  }{m}\frac{1}{V}%
\int_{M}\left\vert \nabla_{b}F\right\vert ^{2}dv+\frac{1}{V}\int_{M}\left\vert
F\right\vert ^{2}dv.
\]
This family of inequalities was proved on $\left(  \mathbb{S}^{2m+1}%
,\theta_{c}\right)  $ by Frank and Lieb \cite{FL} as a corollary of their
sharp Hardy-Littlewood integral inequality on the Heisenberg group. It is
interesting to know if it holds in a more general setting.

A possible approach to the above conjecture is to modify the Jerison-Lee
identity. We hope to report progress on this front in the future.

\end{document}